\documentclass{amsart}
\usepackage{amsmath, amssymb,epic,graphicx,mathrsfs,enumerate}
\usepackage[all]{xy}
\usepackage{color}
\usepackage{comment}

\usepackage{amsthm}
\usepackage{amssymb}
\usepackage{latexsym}
\usepackage{longtable}
\usepackage{epsfig}
\usepackage{amsmath}
\usepackage{hhline}

 \DeclareMathOperator{\perm}{Sym}
 \DeclareMathOperator{\soc}{soc}
\DeclareMathOperator{\aut}{Aut}

 \DeclareMathOperator{\frat}{Frat}

\DeclareMathOperator{\sym}{Sym}

\DeclareMathOperator{\End}{End}

\renewcommand{\emptyset}{\varnothing}

\newtheorem{thm}{Theorem}

 \newtheorem{lemma}[thm]{Lemma}
\newtheorem{prop}[thm]{Proposition} 
 
\newtheorem{question}[thm]{Question}

\numberwithin{equation}{section}

\renewcommand{\footnote}{\endnote}
\newcommand{\ignore}[1]{}\makeglossary

\begin{document}
	\bibliographystyle{amsplain}
	\title{The independence graph of a finite group}

	\author{Andrea Lucchini}
	\address{Andrea Lucchini\\ Universit\`a degli Studi di Padova\\  Dipartimento di Matematica \lq\lq Tullio Levi-Civita\rq\rq\\ Via Trieste 63, 35121 Padova, Italy\\email: lucchini@math.unipd.it}
	

\begin{abstract} Given a finite group $G,$ we denote by $\Delta(G)$ the graph whose vertices are the elements $G$ and where two vertices $x$ and $y$ are adjacent if there exists a minimal generating set of $G$ containing $x$ and $y.$ We prove that $\Delta(G)$ is connected and classify the groups $G$ for which $\Delta(G)$ is a planar graph.	\end{abstract}
	\maketitle

\section{Introduction}

The generating  graph of a finite group $G$ is the graph defined on the elements of $G$ in such a way that two distinct vertices are connected by an edge if and only if they generate $G.$  It was defined by Liebeck and
Shalev in \cite{LS}, and has been further investigated by
many authors: see for example \cite{ bglmn,  bucr, tra, CL3,  fu, diam, LM2, lm, LMRD} for some
of the range of questions that have been considered. Clearly the generating graph of $G$ is an edgeless graph if $G$ is not 2-generated.  We propose and investigate a possible generalization, that gives useful information even when $G$ is not 2-generated.

Let $G$ be a finite group. A generating set $X$ of $G$ is said to be minimal if no proper
subset of $X$ generates $G$.
We denote by $\Gamma(G)$  
the graph whose vertices are the elements $G$ and in which two vertices $x$ and $y$ are joined by an edge if and only if $x\neq y$ and there exists a minimal generating set of $G$ containing $x$ and $y$. Roughly speaking,  $x$ and $y$ are adjacent vertices of $\Gamma(G)$ is they are \lq independent\rq, so we  call $\Gamma(G)$ the {\slshape{independence graph}} of $G$. We will denote by $V(G)$ the set of the non-isolated vertices of $\Gamma(G)$ and by $\Delta(G)$ the subgraph of $\Gamma(G)$ induced by $V(G).$ Our main result is the following.

\begin{thm}\label{main}
	If $G$ is a finite group, then the graph $\Delta(G)$ is connected.
\end{thm}

We prove a stronger result in the case of finite soluble groups. For a positive integer $u$, we denote by $\Gamma_u(G)$ the subgraph of $\Gamma(G)$ in which $x$ and $y$ are joined by an edge if and only if there exists a minimal generating set of size $u$ containing $x$ and $y.$ As before, we denote by $\Delta_u(G)$ the subgraph of $\Gamma_u(G)$ induced by the set $V_u(G)$ of its non-isolated vertices. Notice that, even when $G$ is $u$-generated, the set $V_u(G)$ is in general different from $V(G).$ For example if  $G=\sym(4),$ then $\{(1,2)(3,4), (1,2), (1,2,3)\}$ is a minimal generating set for $G,$ so $(1,2)(3,4)\in V(G);$ however  $(1,2)(3,4)\notin V_2(G).$ If $G$ is a non-cyclic 2-generated group, then $\Gamma_2(G)$ coincides with the generating graph of $G$ and it follows from \cite[Theorem 1]{CL3} that $\Delta_2(G)$ is a connected graph if $G$ is soluble. We generalize this result in the following way.

\begin{thm}\label{conu}
	If $u\in \mathbb N$ and $G$ is a finite soluble group, then $\Delta_u(G)$ is connected.
\end{thm}

Recall that a graph is said to be embeddable in the plane, or planar, if it can be drawn in the plane so that its edges intersect only at their ends. The 2-generated finite groups whose generating graph is planar have been classified in \cite{planar}. Our next result gives a classification of the finite groups $G$ such that $\Gamma(G)$ is a planar graph.

\begin{thm}Let $G$ be a finite group. Then $\Gamma(G)$ is planar if and only either $G\in \{C_2\times C_2, C_2\times C_4, D_4, Q_8, \perm(3)\}$ or $G=C_n$ is cyclic of order $n$ and one of the following occurs:
	\begin{enumerate}
		\item $n$ is a prime-power.
		\item $n=p\cdot q,$ where $p$ and $q$ are distinct primes and $p\leq 3$.
		\item $n=4\cdot q,$ where $q$ is an odd prime.
	\end{enumerate}
\end{thm}

Other results, and some related open questions, are presented in Section \ref{last}.

 \section{Proof of Theorem \ref{main}}
\begin{lemma}
	Let $g\in G.$ Then $g$ is isolated in $\Gamma(G)$ if and only if either $G=\langle g\rangle$ or $g\in \frat(G).$
\end{lemma}
\begin{proof}
Suppose $g\notin \frat(G).$ There exists a maximal subgroup $M$ of $G$ with $g\notin M.$  The set $X=\{g\}\cup M$ contains a minimal generating $X$ of $G$ and $g\in X$ (otherwise $G=\langle X\rangle \leq M$). If $X\neq \{g\},$ then $g$ is not isolated, otherwise $\langle g \rangle=G.$
\end{proof}

\begin{prop}\label{cyc}
If $G$ is a finite cyclic group, then $\Delta(G)$ is connected.
\end{prop}
\begin{proof}
	Let $|G|=p_1^{a_1}\cdots p_t^{a_t},$ where $p_1,\dots,p_t$ are distinct primes. If $t=1,$ then $V(G)=\emptyset.$ So assume $t>1$ and, for $1\leq i\leq t,$ let $g_i$ be an element of $G$ of order $|G|/p_i^{a_i}.$ The subset $X=\{g_1,\dots,g_t\}$ induces a complete subgraph of $\Delta(G).$ Now let $x\in V(G).$ Since $x\notin \frat(G),$ there exists $i\in \{1,\dots,t\}$ such that $p_i^{a_i}$ divides $|x|$ and $x$ is adjacent to $g_i.$
\end{proof}

\begin{lemma}\label{easy}Let $N$ be a normal subgroup of a finite group $G.$
	If $Y=\{y_1,\dots,y_t\}$ has the property that $\langle Y, N\rangle=G$, but $\langle Z, N\rangle \neq G$ for every proper subset $Z$ of $Y,$  then there exist $n_1,\dots,n_u\in N$ such that
	$\{y_1,\dots,y_t,n_1,\dots,n_u\}$ is a minimal generating set of $G.$
\end{lemma}
\begin{proof}
	Since $G=\langle Y, N\rangle,$ $Y\cup N$ contains a minimal generating set $X$ of $G,$ and the minimality property of $Y$ implies $Y\subseteq X.$
\end{proof}

\begin{lemma}\label{facile}Let $N$ be a normal subgroup of a finite group $G.$ If $x_1N$ and $x_2N$ are joined by an edge of $\Delta(G/N),$ then $x_1n_1$ and $x_2n_2$ are joined by an edge of $\Delta(G)$ for every $n_1, n_2\in N.$
	
\end{lemma}
\begin{proof}
Let	$\{x_1N,x_2N,x_3N,\dots,x_tN\}$ be a minimal generating set of $G/N.$ By Lemma \ref{easy}, for every $n_1, n_2 \in N,$ there exists $m_1,\dots,m_u\in N$ such that $$\{ x_1n_1,x_2n_2,x_3,\dots,x_t,m_1,\dots,m_u\}$$ is a minimal generating set of $G.$
\end{proof}

We will  write $x_1\sim_G x_2$ if $x_1$ and $x_2$ belong to the same connected component of $\Delta(G).$ The following lemma is an immediate consequence of Lemma \ref{facile}.

\begin{lemma}\label{basic}
	Let $N$ be a normal subgroup of a finite group $G$ and let $x, y \in G.$ If $xN, yN \in V(G/N)$ and $xN\sim_{G/N}yN,$ then $x\sim_G y.$
\end{lemma}

\begin{lemma}{\cite[Corollary 1.5]{bgk}}\label{spre} Let $G$ be a finite group with $S:= F^*(G)$ nonabelian simple. If $x,y$ are nontrivial
elements of $G,$ then there exists $s\in G$ such that $\langle x, s\rangle$ and $\langle y, s\rangle$ both contain S. 
\end{lemma}

\begin{lemma}
	\label{nonabmon}
	Let $G$ be a finite monolithic primitive group. Assume that $N=\soc G$ is non abelian and that 
	$G= \langle x_1, N\rangle=\langle x_2, N\rangle.$ Then there exists $m\in N$ such that
	$\langle x_1, m \rangle=\langle x_2, m \rangle=G.$
\end{lemma}
\begin{proof}
	We have  $N=S_1\times\dots\times S_t$, where $t\in \mathbb N$ and $S_i\cong S$ with $S$ a nonabelian simple group.
	First consider the case $t=1.$ By Lemma \ref{spre} there exists $m\in N$ with $\langle x_1, m \rangle=\langle x_2, m \rangle=G.$ Assume $t>1.$
	We have $G\leq \aut(S)\wr \perm(t)$ and it is not restrictive to assume $x_1=(h_1,\dots,h_t)\sigma$ with   $h_1,\dots,h_t\in \aut(S),$ $\sigma\in \perm(t)$ and $\sigma(1)=2.$ There exists $u\in \mathbb Z$ such that $x_2^u=(h_1^*,\dots,h_2^*)\sigma,$ with   $h_1^*,\dots,h_t^*\in \aut(S).$ Set $l_1:=x_1,$ $l_2:=x_2^u,$ $k_1:=h_1,$ $k_2:=h_1^*.$ Let $w$ be an element of $S$ of order 2. By Lemma \ref{spre}, there exists $s\in S$ such that $\langle w^{k_1},s\rangle=\langle w^{k_2},s\rangle=S.$ For $1\leq i\leq t,$ consider the projection $\pi_i: N\to S_i\cong S.$ Let
	$m=(w,s,1,\dots,1)\in N\cong S^t.$ 
 For $i\in \{1,2\},$ the subgroup
	$R_i:=\langle m, x_i\rangle$ contains $\langle m, m^{l_i}\rangle \leq N.$ Notice that $S=\langle s, w^{k_i}\rangle\leq \pi_2(\langle m, m^{l_i}\rangle),$ hence $\pi_2(R_i\cap N)\cong S.$
Since $R_iN=G$, we deduce that $\pi_j(R_i\cap N)\cong S$ for each $j\in \{1,\dots,t\}.$ In particular (see for example \cite[Proposition 1.1.39]{classes}) either $N\leq R_i$ or there exist $k\in\{1,\dots,t\}$ and $h \in \aut(S)$ such that $\pi_k(z)=h(\pi_1(z))$ for each $z\in R_i\cap N.$ The second possibility cannot occur, since $m=(w,s,1,\dots,1)\in R_i\cap N$ and $s$ and $w$ are not conjugate in $\aut S$ ($|w|=2$, while $|s|\neq 2$, otherwise $S$ would be generated by two involutions). So $N\leq R_i$ and consequently  $R_i=G.$ 
\end{proof}

\begin{proof}[Proof of Theorem \ref{main}]
	We prove the theorem by induction on the order of $G.$ If can be easily seen that $x\in V(G)$ if and only if $x\frat(G)\in V(G/\frat(G))$ and that $\Delta(G)$ is connected if and only if $\Delta(G/\frat(G)$ is connected. So if $\frat(G)\neq 1,$ the conclusion follows by induction. We may so assume $\frat(G)\neq 1.$ 
Let $N$ be a minimal normal subgroup of $G$ and let $x, y \in V(G)$. If $xN$ and $yN$ are non-isolated vertices of $G/N$, then by induction $xN \sim_G yN$, so it follows from Lemma \ref{basic} that $x\sim_G y.$ This means that the set $\Omega_N$ of the elements
 $g\in V(G)$ such that $gN\in V(G/N)$ is contained in a unique connected component, say $\Gamma_N$, of $\Delta(G).$ Assume now $g\in V(G) \setminus \Omega_N.$ If  $G/N$ is non-cyclic, then $gN\in \frat(G/N).$ In particular a minimal generating set of $G$ containing $g$ must contain also an element $z$ such that $zN\not\in \frat(G/N).$ 
 But then $z\in \Omega_N$ and,
 since $z\in \Gamma_N$ and $g\sim_G z,$ we conclude $g\in \Gamma_N.$
  In other words, if $G/N$ is cyclic, then $\Gamma_N=V(G).$ So we may assume that $G/N$ is cyclic for every minimal normal subgroup $N$ of $G.$
 
  This implies that one of the following occur:
	\begin{enumerate}
		\item $G$ is cyclic;
		\item $G\cong C_p\times C_p$
		\item $G$ has a unique minimal normal subgroup, say $N,$ and $N$ is not central.
	\end{enumerate}
If $G$ is cyclic, then the conclusion follows from Proposition \ref{cyc}. If $G\cong C_p\times C_p$, then $\Delta(G)$ is a complete multipartite graph, with $p+1$ parts of size $p-1.$
So we may assume that the third case occurs. First assume that $N$ is abelian. In this case $N$ has a cyclic complement, $H=\langle h \rangle$, acting faithfully and irreducibly on $N.$ We have $\langle n, h \rangle=G$ for every non trivial element $n$ of $G$, and this implies that there exists a unique connected component $\Lambda$ of $\Delta(G)$ containing all the non trivial elements of $N.$ Let now $g\in G\setminus N.$ There is a conjugate $h^*$ of $h$ in $G$ with $g\notin \langle h^* \rangle.$ If $1\neq n\in N,$ then $G=\langle n, h^*\rangle=
\langle g, h^*\rangle,$ so $g\sim_G h^*\sim_G n$, hence $g\in \Lambda.$
We remain with the case when  $N$ is non-abelian. Let $F/N=\frat(G/N)$ and set $\Sigma_1=F\setminus \{1\}$, $\Sigma_2=\{g\in G\mid \langle g \rangle N=G\},$ $\Sigma_3=\{g\in G\mid gN\in V(G/N)\}$ (we have $\Sigma_3 = \emptyset$ if and only if $|G/N|$ is a prime power). Notice that $V(G)$ is the disjoint union of $\Sigma_1,$ $\Sigma_2$ and $\Sigma_3.$ By Lemma \ref{nonabmon}, all the elements of $\Sigma_2$ belong to the same connected component, say $\Gamma,$ of $\Delta(G).$ Assume $\Sigma_3\neq\emptyset$. Fix $y\in \Sigma_2$ and choose $n$ such that $G=\langle y, n\rangle.$ Let $p$ be a prime divisor of $|G/N|$ and let $y_1, y_2$ be generators, respectively, of a Sylow $p$-subgroup and a $p$-complement of $\langle y \rangle.$ Since $\{y_1,y_2,n\}$ is a minimal generating set for $G$, it follows $y_1, y_2\in \Sigma_3$
and that $y_1\sim_G y_2\sim_G  y \sim_G n.$ But we noticed in the first part of this proof that all the elements of $\Sigma_3=\Omega_N$ belong to the same connected component, and so $\Sigma_2\cup \Sigma_3 \subseteq \Gamma.$ Finally let $g\in \Sigma_1$ and let $X$ be a minimal generating set of $G$ containing $g$. Certainly $X\cap (\Sigma_2\cup \Sigma_3)\neq \emptyset,$ so $g\in \Gamma.$
\end{proof}

\section{Soluble groups}
Let $u$ be a positive integer and $G$ a finite group. In this section
we will use the following notations. We will denote by $\Omega_u(G)$ the set of the minimal generating sets of $G$ of size $u,$ by $\Gamma_n(G)$ the graph whose vertices are the elements of $G$ and in which $x_1$ and $x_2$ are adjacent if and only if there exists $X\in \Omega_n(G)$ with $x_1, x_2 \in X.$
Moreover we will denote by $V_n(G)$ the set of the non-isolated vertices of $\Gamma_u(G)$ and by $\Delta_u(G)$ the subgraph of $\Gamma_u(G)$ induces by $V_u(G).$ Finally we will write $x_1\sim_{G,u}x_2$ to indicate that $x_1$ and $x_2$ belong to the same connected component of $\Delta_u(G).$

We will need a series of preliminary results before giving the proof of Theorem \ref{conu}. The following is immediate.

\begin{lemma}\label{frat}Let $G$ be a finite group. Then $\Delta_u(G)$ is connected if and only if $\Delta_u(G/\frat(G)$ is connected.
\end{lemma}

Given a subset $X$ of a finite group $G,$ we will denote by $d_X(G)$ the smallest cardinality
of a set of elements of $G$ generating $G$ together with the elements of $X.$ 

\begin{lemma}{\cite[Lemma 2]{CL3}}\label{modg} Let $X$ be a subset of $G$ and $N$ a normal subgroup of $G$ and suppose that
	$\langle g_1,\dots,g_r, X, N\rangle=G.$
	If $r\geq d_X(G),$ we can find $n_1,\dots,n_r\in N$ so that $\langle g_1n_1,\dots,g_rn_r,X\rangle=G.$
\end{lemma}

\begin{lemma}\label{indu}
	Let $N$ be a normal subgroup of a finite group group $G$ and consider the projection $\pi:G\to G/N$. Suppose $A\in \Omega_u(G/N)$ and $b \in V_u(G)$ with $bN\in A.$
	Then there exists $B\in \Omega_u(G)$ such that $b\in B$ and $A=\pi(B).$
\end{lemma}
\begin{proof}
	Let  $A=\{bN,z_1N,\dots,z_{u-1}N\}$ and $t=d_{\{b\}}(G).$ Since $b\in V_u(G),$ $t\leq u-1$. 
	By Lemma \ref{modg}, there exist $n_1,\dots,n_{u-1}\in N$ such that
	$\langle b,z_1n_1,\dots,z_{u-1}n_{u-1}\rangle=G$. The set $B:=\{b,z_1n_1,\dots,z_{u-1}n_{u-1}\}$ satisfies the requests of the statement. 
\end{proof}

\begin{lemma}\label{basic}
	Let $N$ be a normal subgroup of a finite group $G$ and let $x, y \in V_u(G).$ If $xN, yN \in V_u(G/N)$ and $xN\sim_{{G/N},u}yN,$ then there exists $n\in N$ such that $x\sim_{G,u} yn.$
\end{lemma}
\begin{proof} Since $xN\sim_{G/N,u}yN,$ there exists a sequence $A_1,\dots,A_t$ of elements of $\Omega_u(G/N)$ such that $xN\in A_1,$ $yN\in A_t$ and $A_i\cap A_{i+1}\neq \emptyset$ for $1\leq i\leq t-1.$ We claim that there exists a sequence $B_1,\dots,B_t$ of minimal generating sets of $G$ such that $x\in B_1,$ $\pi(B_i)=A_i$ for $1\leq i\leq t$ and $B_i\cap B_{i+1}\neq \emptyset$ for $1\leq i\leq t-1.$ By Lemma \ref{indu}, there exists a minimal generating set $B_1$ of $G$ with $A_1= \pi(B_1)$ and $x\in B_1.$ Suppose that $B_1,\dots,B_j$ have been constructed for $j<t.$ There exists $g\in B_j$ such that $gN\in A_j\cap A_{j+1}.$ Again by Lemma \ref{indu},  there exists a minimal generating set $B_{j+1}$ of $G$ with $A_{j+1}=\pi(B_{j+1})$ and $g\in B_{j+1}.$ 
\end{proof}

Denote by $d(G)$ and $m(G)$, respectively, the smallest and the largest cardinality of a minimal generating set of $G.$ A nice result in universal algebra, due to Tarski and known with the name of  Tarski irredundant basis theorem (see for example \cite[Theorem 4.4]{bs}) implies that, for every positive integer $k$ with $d(G)\leq k\leq m(G),$ $G$ contains an independent generating set of cardinality $k.$ The proof of this theorem relies on a clever but elementary counting argument which implies also the following result: 
\begin{lemma}\label{tar}
For every $k$ with $d(G)\leq k<m(G)$ there exists a minimal generating set 
$\{g_1,\dots,g_k\}$ with the property that there are $1\leq i\leq k$ and $x_1,x_2$ in $G$ such that $\{g_1,\dots,g_{i-1},x_1,x_2,g_{i+1},\dots,g_k\}$ is again a minimal generating set of $G.$ Moreover $x_1,x_2$ can be chosen with the extra property that  $g_i=x_1x_2.$
\end{lemma}

Recall that for a $d$-generator finite group $G,$ the swap graph $\Sigma_d(G)$ is the graph in which the vertices are the
ordered generating $d$-tuples and in which two vertices $(x_1,\dots,x_d)$ and $(y_1,\dots,y_d)$ are adjacent if and only if
they differ only by one entry.

\begin{prop}\label{swa}
	Let $G$ be a finite soluble group. Then $\Delta_{d(G)}(G)$ is connected.
\end{prop}
\begin{proof}Let $d=d(G).$
If $G$ is cyclic, then $\Delta_d(G)$ is a null graph, and there is nothing to prove. Assume $d\geq 2$ and let $x, y \in V_d(G).$ Let $X, Y \in \Omega_d(G)$ with $x\in X$ and $y\in Y.$ By \cite{ds}, the swap graph $\Sigma_d(G)$ is connected, so there exists a path in $\Sigma_d(G)$ joining $X$ to $Y.$ Notice that if $A, B$ are adjacent vertices of $\Sigma_d(G),$ then there exists two connected components $\Gamma_A$ and $\Gamma_B$ of $\Delta_d(G)$ containing, respectively,
$A$ and $B.$ On the other hand $A\cap B\neq \emptyset$, by the way in which the swap graph is defined. Thus $\Gamma_A\cap \Gamma_B \neq \emptyset$ and consequently $\Gamma_A=\Gamma_B$ and all the elements of $A\cup B$ belong to the same connected component. This implies in particular that if $A_1=X, A_2,\dots, A_{t-1},A_t=Y$ is a path joining $X$ and $Y$, then all the elements of $\cup_{1\leq i\leq t}A_i$ belong to the same connected component.
\end{proof}

\begin{proof}[Proof of Theorem \ref{conu}]
	We may assume $d(G)\leq u\leq m(G),$ otherwise $V_u(G)$ is empty. If $u=d(G),$ then the results follows from Proposition \ref{swa}.  So we assume $u>d(G).$ We prove the statement by induction on $|G|.$ By Lemma \ref{frat}, we may assume $\frat(G)=1.$
	
	Let $N$ be a minimal normal subgroup of $G.$ 	Let $K$ be a complement of $N$ in $G.$ We have $d(K)\leq d(G)<u$ and $m(K)=m(G/N)=m(G)-1\geq u-1$ (see \cite[Theorem 2]{min}). By the Tarski irredundant basis theorem, $K$ has a minimal generating set $\{k_1,\dots,k_{u-1}\}$ of size $u-1$ and $\{k_1,\dots,k_{u-1},m\}$ is a minimal generating set of $G$ for every $m\neq 1.$ This implies that all the non-trivial elements of $N$ belong to the same connected component, say $\Gamma,$ of $\Delta(G).$ 
	
	In order to complete our proof, we are going to show that $X\cap \Gamma\neq \emptyset,$ for every minimal generating set $X=\{x_1,\dots,x_u\}$ of $G.$  For each $i \in \{1,\dots,u\},$ there exists $k_i\in K$ and $n_i\in N$ such that $x_i=k_in_i.$  We may order the indices in such a way that $Y=\{k_1,\dots,k_t\}$ is a minimal generating set for $K.$
	
	We distinguish two cases.
	
	a) $t < u.$ Let $H=\langle x_1,\dots, x_t\rangle.$ Since $G=\langle Y\rangle N=HN$ and $H\neq G,$ we deduce that
$H$ is a complement for $N$ in $G$ and $\langle H, x_{t+1}\rangle=G.$
In particular $t=u-1$ and $\{x_1,\dots,x_{t},m\}\in \Omega_u(G)$ for every $1\neq m\in G.$
This implies $\{x_1,\dots,x_{t}\}\subseteq \Gamma \cap X.$

b) $t=u.$ Since $d(K)\leq d(G)<u$ and $m(K)\geq u,$ by Lemma \ref{tar} there exists  $\{z_1,\dots,z_u\}\in \Omega_u(K)$ with the property that $\{z_1z_2,\dots,z_u\}\in \Omega_{u-1}(K).$ We first want to prove that if $n\in N$ and $\tilde z:=z_un\in V_n(G),$ then $\tilde z\in \Gamma.$ 
First suppose that there exists a complement $H$ on $N$ in $G$ containing $\tilde z.$ There exist $m_1,\dots,m_{u-1}\in N$ such that $z_im_i \in H$ for $1\leq i\leq u-1.$ This implies $$H=\langle z_1m_1z_2m_2,z_3m_3,\dots,z_{u-1}m_{u-1},\tilde z\rangle,$$ but then $\{z_1m_1z_2m_2,z_3m_3,\dots,z_{u-1}m_{u-1},\tilde z, m\}\in \Omega_u(G)$  for every $1\neq m\in M.$ Thus $\tilde z$ and $m$ are adjacent vertices of $\Delta_u(G)$ and consequently $\tilde z \in \Gamma.$
Now assume that no complement of $N$ in $G$ contains $\tilde z.$ 
If $1\neq m\in N,$ then  $\langle z_1z_2,z_3,\dots,z_u,m\rangle=G,$  hence $d_{\{z_u\}}(G)\leq u-1.$ Since $G=\langle z_1,z_2,\dots,z_u,N\rangle,$ by Lemma \ref{modg}  there exist $m_1,\dots,m_{u-1}\in M$ such that
	$\langle z_1m_1,\dots,z_{u-1}m_{u-1},z_u\rangle = G.$ As before, this implies  $z_u\in \Gamma$ and consequently $\{z_1m_1,\dots,z_{u-1}m_{u-1}\}\subseteq \Gamma.$ On the other hand, since no complement for $N$ in $G$ contains $\tilde z,$ it must be
		$\langle z_1m_1,\dots,z_{u-1}m_{u-1},\tilde z\rangle = G.$ So 
		$\tilde z$ is adjacent to the vertices $z_1m_1,\dots,z_{u-1}m_{u-1}$ of $\Delta_u(G)$ and consequently $\tilde z \in \Gamma.$
	Now we can conclude our proof. Since $\{x_1N,\dots x_uN\}, 
	\{z_1N,\dots z_uN\} \in \Omega_u(G/N),$ we have $x_1N,$ $z_uN\in V_n(G/N)$, so by induction $x_1N \sim_{G/N,u} z_uN.$ By Lemma \ref{basic} there exists $n\in N$ such that
	$x_1\sim_{G,u}  z_un.$ But we proved before that $z_un \in \Gamma,$ and this implies $x_1\in \Gamma \cap X.$
\end{proof}

\section{Planar graphs}

\begin{lemma}\label{plaq}
	Let $N$ be a normal subgroup of a finite group $G$. If $\Gamma(G)$ is planar, then either $G/N$ is cyclic of prime-power order or $|N|\leq 2.$
\end{lemma}
\begin{proof}
	Assume that $G/N$ is not a  cyclic group of prime-power order. Then $\Delta(G/N)$ is not a null-graph. In particular there exist $x$ and $y$ in $G$ such that $xN$ and $yN$ are joined by an edge of $\Gamma(G/N).$ By Lemma \ref{facile}, the subgraph of
	$\Gamma(G)$ induced by $xN\cup yN$ is isomorphic to the complete bipartite graph $K_{a,a},$ with $|a|=N.$ If $\Gamma(G)$ is planar, then $K_{a,a}$ is planar, and this implies $a\leq 2.$
\end{proof}

\begin{prop}
	$\Gamma(C_n)$ is planar if and only if one of the following occurs:
	\begin{enumerate}
		\item $n$ is a prime-power.
		\item $n=p\cdot q,$ where $p$ and $q$ are distinct primes and $p\leq 3$.
		\item $n=4\cdot q,$ where $q$ is an odd prime.
	\end{enumerate}
\end{prop}
\begin{proof}If $n=p^a$ is a prime power, then $\Gamma(C_n)$ is an edgeless graph, and consequently it is planar. Assume that $n$ is not a prime power and let $p < q$ be the two smaller prime divisors of $n.$ We have that $C_n$ contains a normal subgroup $N$ such that $G/N\cong C_{p \cdot q},$ and  it follows from Lemma \ref{plaq} that $|N|\leq 2.$ If $|N|=1,$ then $\Delta(C_n)\cong K_{p-1,q-1}$, and consequently $\Gamma(C_n)$ is planar if and only if $p\leq 3.$ If $|N|=2$, then $p=2$ and $\Delta(G)\cong K_{2,2(q-1)}$, which is a planar graph.
\end{proof}

\begin{lemma}\label{mindue}
	Let $G$ be a finite group. If $G$ is not cyclic, then  there exists a normal subgroup $N$ of $G$ with the property that $d(G/N)=2$ but $G/M$ is cyclic for every normal subgroup $M$ of $G$ with $N<M.$
\end{lemma}
\begin{proof}
Let $\mathcal M$ be the set of the normal subgroups $M$ of $G$ with the property that $d(G/M)=2.$ We claim that if $G$ is not cyclic, then
$\mathcal M \neq \emptyset.$ Indeed let $$1=N_t<\dots <N_0=G$$ be a chief series of $G$ and let $j$ be the smallest positive integer with the property that $G/N_j$ is not cyclic. By  \cite[Theorem 1.3]{old}, 
$d(G/N_j)=2.$ Once we know that $\mathcal M$ is not empty, any subgroup in $\mathcal M$ which is maximal with respect to the inclusion satisfies the requests of the statement.
\end{proof}

\begin{prop}Let $G$ be a finite, non-cyclic group. Then $\Gamma(G)$ is planar if and only if $G\in \{C_2\times C_2, C_2\times C_4, D_4, Q_8, \perm(3)\}$
\end{prop}

\begin{proof}Let $G$ be a non-cyclic group. Choose a normal subgroup $N$ of $G$ as described in Lemma \ref{mindue}. It follows from Lemma \ref{facile} that $\Gamma(G)$ contains a subgraph isomorphic to $\Delta(G/N).$ So if $\Gamma(G)$ is planar, then $\Gamma(G/N)$ is  planar and $|N|\leq 2.$ By \cite{planar} either $G/N\cong C_2\times C_2$ or $G/N\cong \perm(3).$ If $G/N\cong C_2\times C_2$ then either $d(G)=m(G)$ and $G\in \{C_2\times C_2, C_2\times C_4, D_4, Q_8\},$ or
	$d(G)=m(G)=3$ and $G\cong C_2\times C_2\times C_2.$ In the last case $\Delta(G)\cong K_7$ is not planar. In the other cases, $\Gamma(G)$ coincides with the generating graph of $G$ and it is planar. If $G/N \cong \perm(3),$ then  $G\cong \perm(3)$, $G\cong D_6$ or $G\cong C_3\rtimes C_4.$ If $G\cong S_3$ then $\Gamma(G)$ coincides with the generating graph and it is planar. If $G\cong D_6,$ then the six non-central involutions induces a complete subgraph, so $\Gamma(G)$ is not planar. If $G\cong C_3\rtimes C_4,$ then the subset $A\cup B$, where $A$ is the set of the six elements of order 4 and $B$ is the set of the four elements with order divisible by 3, induces a non planar graph containing an isomorphic copy of $K_{6,4}.$
\end{proof}

\section{Examples and questions}\label{last}
 The minimal generating sets for $\perm(4)$ are described in \cite{caca}. We have that $d(\perm(4))=2$ and $m(\perm(4))=3$ and the three graphs $\Gamma_2(\perm(4)),$ $\Gamma_3(\perm(4))$ and $\Gamma(\perm(4))$ are described in the following tables, where the first column contains a representative $x$ of a conjugacy class of $\perm(4)$, the second column describes the set of the elements of $\perm(4)$
adjacent to $x$ in the graph and  the third columns gives the degree of $x$ in the graph. We denote by $X_i$ the set of $i$-cycles (for $2\leq i\leq 4)$ in $\sym(4)$ and by $Y$ the set of the double transpositions.

\begin{table}[h!]
	\begin{center}
\begin{tabular}{|c|c|c|}
	\hline
\multicolumn{3}{|c|}{$\Gamma_2(\perm(4))$}\\
\hline
\hline	
	(1,2)(2,3) & $\emptyset$ &          0 \\
	\hline
	
	(1,2) & $\{(2,3,4)^{\pm 1},(1,3,4)^{\pm 1}, (1,2,3,4)^{\pm 1},(1,2,4,3)^{\pm 1}\}$ &          8 \\
	\hline
	
	(1,2,3) & $X_4\cup \{(1,4), (2,4), (3,4)\}$ &          9 \\
	\hline
	
	(1,2,3,4) & $X_3 \cup \{(1,2),(1,4),(2,3),(3,4),(1,3,2,4)^{\pm 1},(1,2,4,3)^{\pm 1}\}$ &         16 \\
	\hline
	\end{tabular} 
\end{center}
\end{table}

\begin{table}[h!]
\begin{tabular}{|c|c|c|}
	\hline
\multicolumn{3}{|c|}{$\Gamma_3(\perm(4))$}\\	
\hline
\hline
	(1,2)(3,4) & $X_2\cup X_3$ &          14 \\
	\hline
	
	(1,2) & $Y\cup \{(1,2,3)^{\pm 1},(1,2,4)^{\pm 1},(1,3),(1,4),(2,3),(2,4),(3,4)\}$ &          12 \\
	\hline
	
	(1,2,3) & $Y\cup \{(1,2), (1,3), (2,3), (1,2,4)^{\pm 1},(1,3,4)^{\pm 1},(2,3,4)^{\pm 1}\}$ &         12 \\
	\hline
	
	(1,2,3,4) & $\emptyset$ &     0\\
	\hline
	\end{tabular}  
\end{table}

\begin{table}[h!]
\begin{tabular}{|c|c|c|}
	\hline
\multicolumn{3}{|c|}{$\Gamma(\perm(4))$}\\	
\hline
\hline

	(1,2)(3,4)& $X_2\cup X_3$&14 \\
	\hline
	(1,2)&$Y\!\cup\! X_3\!\cup\! \{(1,3),\!(1,4),\!(2,3),\!(2,4),\!(3,4),\!(1,2,3,4)^{\pm 1}\!,(1,2,4,3)^{\pm 1}\}$ &          20 \\
	\hline
	
	(1,2,3)& $Y\cup X_2\cup  X_4\cup \{(1,2,4)^{\pm 1},(1,3,4)^{\pm 1},(2,3,4)^{\pm 1}\}$ &         21 \\
	\hline
	
	(1,2,3,4) & $X_3 \cup \{(1,2),(1,4),(2,3),(3,4), (1,3,2,4)^{\pm 1},(1,2,4,3)^{\pm 1}\}$ &         16\\
	\hline
	\end{tabular}
\end{table}
\noindent  Denote by $\omega(\Gamma)$ the clique number of a graph $\Gamma.$ By \cite[Theorem 1.1]{LM2}, we have
 $\omega(\Gamma_2(\perm(4))=4$ and a maximal clique is
$\{\!(1,2,3,4),(1,2,4,3),(1,3,2,4),(1,2,3)\!\};$  $\omega(\Gamma_3(\perm(4)))=7$  and a maximal clique is $X_2\cup \{(1,2)(3,4)\};$
$\omega(\Gamma(\perm(4)))=11$  and a maximal clique is $X_2\cup \{(1,2)(3,4), (1,2,3), (1,2,4), (1,3,4), (2,3,4)\}.$  However it is not in general true that $\omega(\Gamma_2(\perm(n))\leq \omega(\Gamma_{n-1}(\perm(n))$. Indeed let $n$ be a sufficiently large odd integer. By \cite[Theorem 1]{bla}, $\omega(\Gamma_2(\perm(n))=2^{n-1}$ while by \cite[Theorem 2.1]{caca} a non-isolated vertex of
$\omega(\Gamma_{n-1}(\perm(n))$ is either a transposition or a 3-cycle or a double transposition, so $\omega(\Gamma_{n-1}(\perm(n))\leq \binom n 2+2\cdot \binom n 3+ 3\cdot \binom n 4.$
The independence number of $\Gamma_2(\perm(4))$ is 12 and a maximal independent set is $X_3\cup Y \cup\{id\}.$ The independence number of $\Delta_3(\perm(4))$ is 8 and a maximal independent set is $X_4\cup\{(1,2),(1,3,4),(1,4,3),id\}.$  The independence number of $\Delta(\perm(4))$ is 6 and a maximal independent set is $Y\cup\{(1,2,3,4),(1,4,3,2),id\}.$ 
For $u\in \{1,2\},$ the degree of the vertex of $\Gamma_u(\perm(4))$ corresponding to the element $g$ is divisible by the order of $g.$ When $u=2,$ this follows from a more general result. Indeed, by \cite[Proposition 2.2]{euler}, if $G$ is a 2-generated group and $g\in G$, then $|g|$ divides the degree of $g$ in the generating graph of $G.$ However this cannot be generalized to $\Gamma_u(G)$ for arbitrary values of $u.$ For example, consider the dihedral group $G=\langle a, b \mid a^6, b^2, (ab)^3\rangle$ of degree 6. Then $\{a^2,a^3,a^ib\}\in \Omega_3(G)$ for $0\leq i\leq 5$ and there are precisely 7 elements adjacent to $a^2$ in $\Gamma_3(G)$: $a^3$ and $a^ib$ for $0\leq i\leq 5.$ We propose the following question.
\begin{question} Let $G$ be a finite group and  $g\in G.$ Does $|g|$ divide the degree of $g$ in $\Gamma_{d(G)}(G)?$
\end{question}
For a finite group $G$, let 
$$W(G)=\bigcap_{d(G)\leq u\leq m(G)}V_u(G).$$
We have seen that $W(\perm(4))=X_2\cup X_3\neq V(\perm(4)).$ 
If $d(G)=m(G),$ then $V(G)=W(G)$ by definition. One may ask whether the converse is true.
\begin{question}\label{vw}
Does  $V(G)=W(G)$ imply $d(G)=m(G)$?
\end{question}
The answer is positive in the soluble case.
\begin{prop}
	Let $G$ be a finite soluble group. If $V(G)=W(G),$ then $d(G)=m(G).$
\end{prop}
\begin{proof}
Let $d=d(G),$  $m=m(G).$ Since $V_u(G/\frat(G))=V_u(G)\frat(G)/\frat(G),$ we may assume $\frat(G)=1.$ First assume that $G$ is cyclic. Then $|G|=p_1\cdots p_m,$ with $p_1,\dots,p_m$ distinct primes. Notice that $V_1(G)=\emptyset$, so $V(G)=W(G)\subseteq V_1(G)$ implies $V(G)=\emptyset$ and this is possible only if $m=1.$ Now assume that $G$ is not cyclic. By assumption $V_d(G)=V(G)=G\setminus \{1\}.$ This is equivalent to say that $d_{\{g\}}(G)=d-1$ for any $1\neq g\in G.$ By \cite[Corollary 2.20, Theorem 2.21]{tra} either $G$ is an elementary abelian $p$-group or there
exist a finite vector space $V$, a nontrivial irreducible soluble subgroup $H$ of $\aut(V)$
and an integer $d>d(H)$ such that
$$G\cong V^{r(d-2)+1} \rtimes H,$$
where $r$ is the dimension of $V$ over $\End_H(V)$ and $H$ acts in the same way on
each of the $r(d-2)+1$ factors. In the first case $d=m$, and we are done.  In the second case, by \cite[Theorem 2]{min}, $m=r(d-2)+1+m(H).$ If $d=2,$ then $H=\langle h\rangle$ is a cyclic group and $G=V\rtimes H.$ Since $H$ is a maximal subgroup, if $h\in \Omega_u(G),$ then $u=2.$ On the other hand, by assumption, $h\in V_m(G)$, and therefore $m=2.$ Assume $d>2.$ This implies $t=r(d-2)+1\geq 2.$ We are going to prove that  $r=1.$ If $r\neq 1,$ then there exist $v_1, v_2\in V$ that are $\End_G(V)$-linearly independent. This implies that the $H$-submodule $W$ of $V^t$ generated by  $w=(v_1,v_2,1,\dots,1)$ is $H$-isomorphic to $V^2.$ As a consequence, if $w\in \Omega_u(G),$ then $u-1\leq m(G/W)=m-2.$ But then $w\notin \Omega_m(G)$, against the assumption $V_m(G)=V(G).$ So $r=1,$ and this implies that $H$ is isomorphic to a subgroup of the multiplicative group of $\End_G(V)$, and consequently it is cyclic. Moreover $t=d-1$ and $m(G)=m(H)+d-1.$ Let $h$ be a generator for of $H.$ Notice that $h\notin \Omega_u(G)$ if $u-1>t=d-1.$ Since, by assumption, $h\in \Omega_u(G),$ it must be $m-1\leq d-1,$ and consequently $m=d.$
\end{proof}
 The finite groups with $d(G)=m(G)$ are described in \cite{ak}. All the finite groups with this property are soluble. So Question \ref{vw} is equivalent to the following.
\begin{question}
Does there exist an unsoluble group $G$ with $V(G)=W(G)?$
\end{question}

Another question that we propose is the following.
\begin{question}
Let $G$ be a finite non-cyclic group. Is the graph $\Delta(G)$ Hamiltonian?
\end{question}
Notice that if $G$ is cyclic, then $\Delta(G)$ is not necessarily Hamiltonian. For example, if $G\cong C_{2\cdot p},$ with $p$ and odd prime, then $\Delta(G)\cong K_{1,p-1}.$
In the case of $\perm(4)$ the affirmative answer to the previous question follows from the Dirac's Theorem, stating that an $n$-vertex graph in which each vertex has degree at least $n/2$ must have a Hamiltonian cycle. However it is not in general true that any vertex of $\Delta(G)$ has degree at least $|V(G)|/2.$ Consider for example $G=\langle a, b \mid a^5, b^4, b^{-1}aba^3\rangle.$ The graph $\Delta(G)$ has 19 vertices, and the degree of $b^2$ in this graph is 8. In any case, we may use Dirac's Theorem in the case of finite nilpotent groups. 

\begin{thm}
	If $G$ is a finite non-cyclic nilpotent group, then $\Delta(G)$ is Hamiltonian.
\end{thm}
\begin{proof}
 Let $g\in V(G)$ and let $H=\langle g\rangle\frat(G).$ Let $n=|V(G)|$ and $d$ the degree of $g\in \Delta(G).$
 Since $G/\frat(G)$ is a direct product of elementary abelian $p$-groups, any element of $G\setminus H$ is adjacent to $g$ in $\Delta(G).$ 
Since $G$ is not cyclic, $H$ is a proper subgroup of $G$, hence $|G|\geq 2|H|$ and therefore
$$d=|G|-|H| \geq \frac{|G|-|\frat(G)|}{2}=\frac{n}{2}.$$ So the conclusion follows from Dirac's theorem.
\end{proof}

Finally, a question that remains open is whether Theorem \ref{conu} remains true is the solubility assumption is removed. 

\begin{question}
Let $G$ be a finite group and $u\in \mathbb N.$  Is $\Delta_u(G)$ a connected graph?
\end{question}

\end{document}